\documentclass[a4paper,reqno, 11pt]{amsart}
\usepackage{amsmath,amssymb,amsthm,enumerate}

\usepackage{paralist}
\usepackage{graphics}
\usepackage{epsfig}
\usepackage{amsfonts}
\usepackage{amssymb}
\usepackage{color}

\usepackage{mathtools}

\usepackage[numbers,sort&compress]{natbib}

\usepackage{mathrsfs}
\usepackage{lineno}

\usepackage[svgnames]{xcolor}

\usepackage{pifont}
\usepackage{mathtools}

\usepackage[misc]{ifsym}

\usepackage[margin=0.8in]{geometry}
\def\ifl{\iffalse }

\def\bc{\begin{center}} \def\ec{\end{center}}
\def\ba{\begin{array}} \def\ea{\end{array}}
\def\bea{\begin{eqnarray}} \def\eea{\end{eqnarray}}
\def\beaa{\begin{eqnarray*}} \def\eeaa{\end{eqnarray*}}

\newtheorem{thm}{Theorem}[section]

\newtheorem{lem}{Lemma}[section]
\theoremstyle{remark}
\newtheorem{rem}{Remark}[section]
\newtheorem*{rem*}{Remark}

\numberwithin{equation}{section}

\newcommand{\R}{\mathbb{R}}

\newcommand{\pa}{\partial}
\newcommand{\na}{\nabla}
\newcommand{\al}{\alpha}

\newcommand{\OK}{\operatorname{OK}}

\newcommand{\Flor}[1]{\lfloor{#1}\rfloor}

\title{Global wellposedness for 2D quasilinear
wave without Lorentz}

\author[X.Y. Cheng]{Xinyu Cheng}
\address{X.Y. Cheng, Department of Mathematics, University of British Columbia, Vancouver, BC V6T 1Z2, Canada}
\email{xycheng@math.ubc.ca}

\author[D. Li]{ Dong Li}
\address{D. Li, Department of Mathematics, the Hong Kong University of Science \& Technology, Clear Water Bay, Kowloon, Hong Kong}
\email{mpdongli@gmail.com}

\author[J. Xu]{Jiao Xu}	
\address{J. Xu, SUSTech International Center for Mathematics, Southern University of Science and Technology,
	Shenzhen, P.R. China}
\email{xuj7@sustech.edu.cn}

\author[D.B. Zha]{Dongbing Zha}
\address{D.B. Zha,  Department of Mathematics and Institute for Nonlinear Sciences, Donghua University, Shanghai, P.R. China}
\email{ZhaDongbing@dhu.edu.cn, ZhaDongbing@163.com}
\begin{document}

\begin{abstract}
We consider the two-dimensional quasilinear wave equations with standard null-form
type quadratic nonlinearities. We prove global wellposedness without using the
Lorentz boost vector fields.
\end{abstract}

\maketitle

\section{introduction}
Denote $\square = \partial_{tt} - \Delta$ as  the usual wave operator. We consider the Cauchy problem for the following two-dimensional quasilinear wave equation:
\begin{align}\label{eq:we2d}
 \begin{cases}
  \square u=g^{kij} \partial_k u \partial_{ij} u,\, \quad t>2, \quad x\in\R^2,\\
  u|_{t=2}=\varepsilon f_1,\ \  \pa_{t} u|_{t=1}=\varepsilon f_2.
 \end{cases}
\end{align}
Here and throughout this note we adopt the  Einstein summation convention with $\partial_0 =\partial_t$ and $\partial_l=\partial_{x_l}$ for $l=1,2$.  For simplicity we assume $g^{kij}$ are constant coefficients,
$g^{kij}=g^{kji}$ for any $i$, $j$,  and satisfy the standard null condition:
\begin{align} \label{null1}
g^{kij} \omega_k \omega_i \omega_j=0, \quad \text{for any null $\omega$,
i.e. $\omega=(-1,\cos\theta, \sin \theta)$, $\theta \in [0,2\pi]$.}
\end{align}
In \cite{Alinh01_1}  Alinhac showed that under the general null condition
\eqref{null1} the system \eqref{eq:we2d} has small data global wellposedness, and  the highest
norm of the solution grows  at most polynomially in time. Alinac's proof relies on the construction of an approximation solution, combined with a judiciously chosen time-dependent weighted energy estimate known since then as the ghost weight method. The energy estimate used therein involves a collection
of vector fields which are well adapted to the d'Alembertian operator. In particular, in order to
harness sufficient time-decay of the solution, the Lorentz boost vector fields were used
heavily in conjunction with the scaling operator.  Heuristically, the benefit of the Lorentz boost can be seen from the following identities
(below $\Omega_{i0}=t \partial_i + x_i \partial_t$ denotes the usual Lorentz boost):
\begin{align}
L_0  = t \partial_t + r \partial_r, \qquad \frac{x_i} r \Omega_{i0} = r \partial_t + t \partial_r.
\end{align}
Clearly away from the light cone (i.e. $r\le t/2$ or $r\ge 2t$), we have $\partial_t $ and $\partial_r
\approx \frac 1 t ( O(L_0)  +O(\Omega_{i0}))$ which readily leads to time-decay estimates. Whilst the Lorentz boost can produce
strong decay estimates, they are not suitable for general wave systems which are not Lorentz invariant.  Such systems include non-relativistic wave systems with multiple wave speeds 
(cf. \cite{SK96,ST01}), nonlinear wave equations
on non-flat space-time (cf. \cite{Yang16}) and exterior domains (cf. \cite{MetSog07}).
From this perspective it is of fundamental importance to remove the Lorentz boost operator
and develop a new strategy for the general non-Lorentz-invariant systems. 
In \cite{Hosh06}, Hoshiga considered a quasilinear system with multiple speeds of propagation,
and proved global wellposedness under some suitable null conditions. A notable novelty in
\cite{Hosh06} is an $L^{\infty}$-$L^{\infty}$ estimate which relies on the fundamental
solution of the wave equation. In \cite{Zha16} (see also  \cite{PZ16}), Zha considered 
\eqref{eq:we2d}--\eqref{null1} with the following additional symmetry condition:
\begin{align} \label{symm1}
g^{kij}= g^{ikj}=g^{jik}, \qquad\forall\, i,j,k.
\end{align}
For this case Zha proved the global wellposedness without using the Lorentz boost vector fields. 
Note that the condition \eqref{symm1} is a bit too restrictive. For example, it does not include
the standard nonlinearity $\partial ( |\partial_t u|^2 - |\nabla u|^2)$. In recent \cite{CLX},
the first three authors introduced a novel strong null form which includes several prototypical
strong null forms such as $\partial ( |\partial_t u|^2 - |\nabla u|^2)$ in the literature as special cases. Moreover   a new normal-form type
strategy was developed in \cite{CLX} 
to prove uniform boundedness of highest norm of the solution. Other related developments with
different strategies 
can be found in the papers \cite{Lei16, CLM18, Cai21, Dong21, HY20, Zh19, Ka17}.

The purpose of this note is to develop further the program initiated in 
\cite{Zha16, PZ16, CLX}, and obtain a full wellposedness result under the standard null
condition \eqref{null1} without employing the Lorentz boost vector fields. Thanks to the aforementioned developments, it is now possible to build a robust and
streamlined Lorentz-free framework for general quasilinear equations. Our main result reads as follows.
\begin{thm}\label{thm:main}
Consider \eqref{eq:we2d} with $g^{kij}$ satisfying the standard null condition
\eqref{null1}. Let $m\ge 5$ and assume $f_1 \in H^{m+1}(\mathbb R^2) $, $f_2\in H^{m} (\mathbb R^2)$ are compactly supported in the disk $\{|x| \le 1\}$.  There exists $\varepsilon_{0}>0$ depending on $g^{kij}$ and $\|f_1\|_{H^{m+1}}+\|f_2\|_{H^{m}}$ such that for all $0\le \varepsilon <\varepsilon_0$, the system \eqref{eq:we2d} has a unique global solution.
Furthermore, the highest norm of the solution is polynomially bounded in time, and the
 second highest norm of the solution remains uniformly bounded, namely
\begin{align}
\sup_{t\ge 2} E_{m-1} (u(t,\cdot))=
\sup_{t\ge 2} \sum_{|\alpha|\le m-1}  \| (\partial \Gamma^{\alpha} u)(t,\cdot) \|^2_{L_x^2(\mathbb R^2)}
< \infty.
\end{align}
Here $\Gamma=\{ \partial_t, \partial_{x_1}, \partial_{x_2}, \partial_{\theta}, t\partial_t +r\partial_r\}$ does not include the Lorentz boost (see  \eqref{def_Gamma1} for notation).
\end{thm}

We now outline the key steps of the proof of Theorem \ref{thm:main} (see section 2
for the relevant notation). To elucidate the idea, we fix any multi-index $\alpha$ and denote $v= \Gamma^{\alpha} u$ (we suppress the dependence on $\alpha$ for simplicity of notation).   By Lemma \ref{Lem2.3}, we have
\begin{align} \label{1.6AA}
\square v = \sum_{\alpha_1+\alpha_2 \le \alpha}  g^{kij}_{\alpha_1,\alpha_2} 
\partial_k \Gamma^{\alpha_1}u \partial_{ij} \Gamma^{\alpha_2} u.
\end{align}
In the forthcoming energy estimates, we shall sum over $|\alpha|\le m_0$, where $m_0$
is a running parameter. If $m_0=m-1$, we seek to show the uniform-in-time boundedness
of $E_{m-1}(u(t,\cdot))$. If $m_0=m$, we show $t^{\tilde \epsilon}$ ($\tilde \epsilon$ is a small
exponent) growth of $E_m(u(t,\cdot))$. 

Step 1. Weighted energy estimates: LHS of \eqref{1.6AA}. We choose $p (r,t)=q(r-t)$ with $q^{\prime}(s)$ nearly
scales as $\langle s \rangle^{-1}$ to derive
\begin{align}
\int \square v \partial_t v e^p dx = 
\frac 12 \frac d {dt}
( \| e^{\frac p2} \partial v \|_2^2)
+ \frac 12 \int e^p q^{\prime} |Tv|^2 dx.
\end{align}
In more detail, we have
\begin{align}
&\sum_{|\alpha| \le m_0} \| e^{\frac p2} \partial v \|_2^2 \sim E_{m_0} 
=\sum_{|\alpha|\le m_0} \| \partial \Gamma^{\alpha} u(t,\cdot) \|_{2}^2; \\
&\sum_{|\alpha|\le m_0}  \int e^p q^{\prime} |Tv|^2 dx
= \sum_{|\alpha|\le m_0} 
\int e^p q^{\prime} |T \Gamma^{\alpha} u |^2 dx. 
\end{align}
When carrying out the energy estimates, we shall use the convention \eqref{No3.5}. 

Step 2. Refined decay estimates. To remedy the lack of Lorentz boost vector fields, 
one has to employ $L^{\infty}$ and $L^2$ estimates involving the weight-factor $\langle
r-t\rangle$. At the expense of certain smallness of $E_{\Flor{\frac m2}+3}$ and
using in an essential way the nonlinear null form
(see Lemma \ref{lem2.3a}), we obtain
\begin{align}
& \| \langle r -t\rangle (\partial^2 \Gamma^{\le l_0}
u)(t, \cdot)\|_2 \lesssim \| (\partial \Gamma^{\le l_0+1} u )(t,\cdot) \|_2,\qquad
\forall\, l_0\le m-1; \\
& |\langle r -t \rangle (\partial^2 \Gamma^{\le l_0}
u)(t,x)| \lesssim | (\partial \Gamma^{l_0+1} )(t,x)|, \qquad\forall\, r\ge t/10, \,
l_0\le m-1; \\
& \| (\partial  \Delta \Gamma^{\le m-3} u )(t, \cdot )\|_{L_x^2(|x|\le \frac 23 t)} \lesssim
t^{-2} \| (\partial \Gamma^{\le m-1} u)(t,\cdot) \|_2. 
\end{align} 
These in turn lead to a handful of strong decay estimates (see Lemma 
\ref{lem2.6}):
\begin{align}
&t^{\frac 12}
\| \partial \Gamma^{\le m-3} u\|_{\infty}
+t^{\frac 32} \| \frac {T\Gamma^{\le m-3} u} {\langle r -t\rangle
}\|_{\infty}
+t^{\frac 12} \| \frac {T \Gamma^{\le m-2} u } {\langle r -t\rangle}
\|_2 \lesssim E_{m-1}^{\frac 12} ;\\
& t^{\frac 12}
\| \langle r-t\rangle \partial^2 \Gamma^{\le m-4} u \|_{\infty}
+ t^{\frac 32}
\| T \partial \Gamma^{\le m-4} u \|_{\infty}
+ t \| T \partial \Gamma^{\le m-4} u \|_2 \lesssim E_{m-1}^{\frac 12}.
\end{align}
These decay estimates play an important role in the nonlinear energy estimates. 

Step 3. Weighted energy estimates: nonlinear terms. We discuss several cases.

Case 1: { $\alpha_1<\alpha$ and $\alpha_2<\alpha$}.
Since $g^{kij}_{\alpha_1,\alpha_2}$ still
satisfies the null condition, by Lemma \ref{Lem2.3} we rewrite
\begin{align}
  &\sum_{\substack{\alpha_1<\alpha,  \alpha_2 <\alpha \\ \alpha_1+\alpha_2\le \alpha}}
g^{kij}_{\alpha_1,\alpha_2} \partial_k \Gamma^{\alpha_1} u
\partial_{ij} \Gamma^{\alpha_2} u    \notag \\
= &\sum_{\substack{\alpha_1<\alpha,  \alpha_2 <\alpha \\ \alpha_1+\alpha_2\le \alpha}}
g^{kij}_{\alpha_1,\alpha_2}
(T_k \Gamma^{\alpha_1} u \partial_{ij} \Gamma^{\alpha_2} u
-\omega_k \partial_t \Gamma^{\alpha_1 } u T_i \partial_j \Gamma^{\alpha_2} u
+ \omega_k \omega_i \partial_t \Gamma^{\alpha_1} u T_j
\partial_t \Gamma^{\alpha_2} u ).
\end{align}

By using the decay estimates proved in Step 2, we show that
\begin{align} 
  &\| \sum_{|\alpha|\le m_0} \sum_{\substack{\alpha_1<\alpha,  \alpha_2 <\alpha \\ \alpha_1+\alpha_2\le \alpha}}
g^{kij}_{\alpha_1,\alpha_2} \partial_k \Gamma^{\alpha_1} u
\partial_{ij} \Gamma^{\alpha_2} u    \|_2 \lesssim  
t^{-\frac 32} E_{\Flor{\frac {m_0}2}+3}^{\frac12} E_{m_0}^{\frac 12}.
\end{align}

Case 2: The quasilinear piece $\alpha_1=0$, $\alpha_2=\alpha$. 
By using successive integration by parts, we have
\begin{align}
 \int g^{kij}\pa_{k}u\pa_{ij}v\pa_{t}v e^{p}
= \OK,
\end{align}
where $\OK$ is in the sense of \eqref{No3.5}. Here a crucial observation is the algebraic
identity (see \eqref{varphi1})
\begin{align}\notag
   &-\pa_{j}\varphi\pa_{i}v\pa_{t}v+\pa_{t}\varphi\pa_{i}v\pa_{j}v-\pa_{i}\varphi\pa_{t}v\pa_{j}v
   \notag\\
   =&-T_{j}\varphi\pa_{i}v\pa_{t}v+\pa_{t}\varphi T_{i}v T_{j} v-T_{i}\varphi\pa_{t}v\pa_{j}v-\omega_{i}\omega_{j}\pa_{t}\varphi(\pa_{t}v)^2,
 \end{align}
where we take $\varphi=\pa_{k}u$ or $\varphi=e^{p}$. The standard null form condition amounts
to the annihilation of the term $\omega_i \omega_j \omega_k$ when
$\varphi =\partial_k u$ and  $\partial_t \varphi$ is
replaced by $T_k\partial_t u -\omega_k \partial_{tt} u$.

Case 3: the main piece $\alpha_1=\alpha$, $\alpha_2=0$. By using Lemma \ref{Lem2.3}
with the decay estimates, we derive
\begin{equation}
  \int g^{kij} \pa_{k}v\pa_{ij}u\pa_{t}v e^{p}= \OK+
  \underbrace{\int g^{kij} \omega_{i}\omega_{j}T_{k}v\pa_{tt}u\pa_{t}v e^{p}}_{=:Y_1}.
\end{equation}
We then discuss two sub-cases. If $m_0=m-1$, we show
\begin{align}
|Y_1| \lesssim t^{-\frac 32} E_3^{\frac 12} E_m^{\frac 12} E_{m_0}^{\frac 12}.
\end{align}
If $m_0=m$, we use Cauchy-Schwartz to bound $Y_1$ as
\begin{align}
|Y_1| &\le \OK + \mathrm{const} \cdot \int \frac 1 {q^{\prime}}|\partial_{tt} u |^2
|\partial_t v|^2 dx \le  \OK+ \frac 1 t E_3 E_m.
\end{align}

Collecting all the estimates and assuming smallness of the initial data, we finally obtain
\begin{align}
\sup_{t\ge 2} E_{m-1} (u(t,\cdot) ) \le \epsilon_3 \ll 1, \qquad
\sup_{t\ge 2} \frac {E_m(u(t,\cdot) )} {t^{\epsilon_4} } \le 1,
\end{align}
where $\epsilon_3>0$, $\epsilon_4>0$ are small constants. This concludes the
proof of Theorem \ref{thm:main}.
\begin{rem}
At this point, it is worthwhile pin-pointing exactly where the symmetry condition $g^{kij}=g^{ikj}=g^{jik}$
for all $k$, $i$, $j$ was used in \cite{PZ16}. In our notation, this comes from bounding the quasilinear piece $\alpha_2=\alpha$. Namely
\begin{align}
 \int g^{kij} \partial_k u \partial_{ij} v \partial_t v e^p dx  &=\int g^{kij} \partial_j ( \partial_k u \partial_i v \partial_t v e^p)
-\int g^{kij} \partial_{kj} u \partial_i v \partial_t v e^p \notag \\
& \qquad -\int g^{kij} \partial_k u \partial_i v \partial_{tj} v e^p
-\int g^{kij} \partial_k u \partial_i v\partial_t v \partial_j (e^p).
\end{align}
By using the symmetry $g^{kij}=g^{kji}$ which is harmless, we have
\begin{align}
- \int g^{kij} \partial_k u \partial_i v \partial_{tj} v e^p
= -\frac 12 \int g^{kij} \partial_t (\partial_ku \partial_i v \partial_j v e^p)
+\frac 12 \int g^{kij} \partial_{tk} u \partial_i v \partial_j v e^p
+ \frac 12 \int g^{kij} \partial_i v \partial_j v \partial_t (e^p).
\end{align}
It follows that
\begin{align}
 \int g^{kij} \partial_k u \partial_{ij} v \partial_t v e^p dx  &=
 -\int g^{kij} \partial_{kj} u \partial_i v \partial_t v e^p
 + \frac 12 \int g^{kij} \partial_{tk} u \partial_i v \partial_j v e^p
 + \frac 12 \int g^{kij} \partial_k u
 \partial_i v \partial_j v \partial_t (e^p)+ \cdots,  \label{intro_rem6a}
 \end{align}
 where $\cdots$ denotes harmless terms.
The second term on the RHS of \eqref{intro_rem6a} is not a problem
thanks to the good decay of $\partial_{tk} u$. On the other hand, in \cite{PZ16} the time-decay
of $\partial^2 u$ in the regime $r\le t/2$ was not sufficient to treat the first
term on the RHS of \eqref{intro_rem6a}. For this reason (see (3.11) in \cite{PZ16}),
Peng and Zha made use of the other piece corresponding to $\alpha_1=\alpha$
and the symmetry $g^{kij}=g^{ikj}$ to eliminate the above term, i.e.:
\begin{align}
\int g^{kij} \partial_k v \partial_{ij} u \partial_t v e^p
= \int g^{kij} \partial_i v \partial_{kj} u \partial_t v e^p.
\end{align}
One of the main novelty of this work is that we obtained $t^{-\frac 32} $ decay of 
$\partial^2 u$
in the regime $r\le t/2$. This and several other new estimates   can have useful applications
in many other problems.
\end{rem}

The rest of this note is organized as follows. In Section 2 we collect some preliminaries and useful
lemmas. In Section 3 we give
 the proof of Theorem \ref{thm:main}.

\subsection*{Acknowledgement.}
D. Li is supported in part by Hong Kong RGC grant GRF 16307317 and 16309518.
D. Zha is supported in part by
National Natural Science Foundation of China No. 11801068 and the Fundamental Research Funds for the Central Universities.
\section{Preliminaries}

\subsection*{Notation}
We shall us the Japanese bracket notation: $ \langle x \rangle = \sqrt{1+|x|^2}$, for $x \in \mathbb R^d$.  We denote  $\partial_0 = \partial_t$,
$\partial_i = \partial_{x_i}$, $i=1,2$ and (below $\partial_{\theta}$ and
$\partial_r$ correspond to the usual polar coordinates)
\begin{align}
& \partial = (\partial_i)_{i=0}^2, \; \partial_{\theta}=x_1 \partial_2 -x_2\partial_1,
\; L_0 = t \partial_t + r \partial_r, \\
& \Gamma= (\Gamma_i )_{i=1}^5, \quad\text{where } \Gamma_1 =\partial_t, \Gamma_2=\partial_1,
\Gamma_3= \partial_2, \Gamma_4= \partial_{\theta}, \Gamma_5=L_0;
\label{def_Gamma}\\
& \Gamma^{\alpha} =\Gamma_1^{\alpha_1} \Gamma_2^{\alpha_2}
\Gamma_3^{\alpha_3} \Gamma_4^{\alpha_4}\Gamma_5^{\alpha_5}, \qquad \text{$\alpha=(\alpha_1,\cdots, \alpha_5)$ is a multi-index}; \label{def_Gamma1}\\
& \partial_+ =\partial_t + \partial_r, \qquad \partial_- =\partial_t - \partial_r; \\
& T_i = \omega_i \partial_t + \partial_i, \; \omega_0=-1, \; \omega_i=x_i/r, \, i=1,2.
\label{2.5a}
\end{align}
Note that in \eqref{def_Gamma} we do not include the Lorentz boosts. Note that $T_0=0$.
For simplicity of notation,
we define for any integer $k\ge 1$,  $\Gamma^k = (\Gamma^{\alpha})_{|\alpha|=k}$,
$\Gamma^{\le k} =(\Gamma^{\alpha})_{|\alpha|\le k}$.
In particular
\begin{align}
|\Gamma^{\le k} u | = (\sum_{|\alpha|\le k} |\Gamma^{\alpha} u |^2)^{\frac 12}.
\end{align}
Informally speaking, it is useful to think of  $\Gamma^{\le k} $ as any one of the vector
fields $ \Gamma^{\alpha}$ with $|\alpha| \le k$. For integer $J\ge 3$, we shall denote 
\begin{align}
E_J = E_J(u(t,\cdot)) = \| (\partial \Gamma^{\le J} u)(t,\cdot) \|_{L_x^2(\mathbb R^2)}^2.
\end{align}

For any two quantities $A$, $B\ge 0$, we write  $A\lesssim B$ if $A\le CB$ for some unimportant constant $C>0$.
We write $A\sim B$ if $A\lesssim B$ and $B\lesssim A$. We write $A\ll B$ if
$A\le c B$ and $c>0$ is a sufficiently small constant. The needed smallness is clear from the context.

 \begin{lem}[Sobolev and Hardy]\label{lem:S}
For $v\in C_c^{\infty}(\R^2)$, we have
\begin{align}
|v(x)| \lesssim \langle x \rangle^{-\frac 12} ( \| \partial_r^{\le 1} \partial_{\theta}^{\le 1}
 v\|_2 + \| \Delta v \|_2).
 \end{align}
 Suppose $u=u(t,x)$ ($t\ge 0$) is smooth and compactly supported in $ \{(t,x):\, |x|\le 1+t\}$, then
 \begin{align}
 \| \langle |x|-t\rangle^{-1} u \|_{L_x^2(\mathbb R^2)}
 \lesssim \| \partial_r u \|_{L_x^2(\mathbb R^2)},
 \qquad \langle |x| -t\rangle^{-1} |u(t,x)| \lesssim 
 \langle x \rangle^{-\frac 12} \| \partial \Gamma^{\le 1} u \|_{L_x^2(\mathbb R^2)}.
 \end{align}
\end{lem}
\begin{proof}
See Lemma 2.1 and 2.2 of \cite{CLX}.
\end{proof}

\begin{lem} \label{Lem2.3}
If $g^{kij}$ satisfies the null condition, then for $t> 0$  we have
\begin{align} \label{2.10A}
g^{kij} \partial_k f
\partial_{ij} h   = g^{kij}
(T_k  f \partial_{ij} h
-\omega_k \partial_t f T_i \partial_j h
+ \omega_k \omega_i \partial_t f T_j
\partial_t h),
\end{align}
where $T=(T_1,T_2)$ is defined in \eqref{2.5a}.  It follows that
\begin{align}
|g^{kij} \partial_k f \partial_{ij} h|
& \lesssim | T f | |\partial^2 h| + |\partial f | | T \partial h|  \label{a2.12a}\\
& \lesssim
\frac 1 {\langle r +t \rangle}
(|\Gamma f| |\partial^2 h| + |\partial f | |\Gamma \partial h|
+ |\partial f | \cdot |\partial^2 h| \cdot |r -t | ). \label{a2.12b}
\end{align}
Suppose $g^{kij}$ satisfies the null condition and
$
\square u = g^{kij} \partial_k u \partial_{ij} u.
$
Then for any multi-index $\alpha$, we have
\begin{align} \label{a2.12aa}
\square \Gamma^{\alpha} u  = \sum_{\alpha_1+ \alpha_2 \le \alpha}
g^{kij}_{\alpha_1,\alpha_2} \partial_k \Gamma^{\alpha_1} u
\partial_{ij} \Gamma^{\alpha_2} u,
\end{align}
where for each ($\alpha_1$, $\alpha_2$), $g^{kij}_{\alpha_1,\alpha_2}$ also satisfies the
null condition.  In addition, we have $g^{kij}_{\alpha,0} =g^{kij}_{0,\alpha}=g^{kij}$.
\end{lem}
\begin{proof}
See Lemma 2.3 of \cite{CLX}.
\end{proof}

\begin{lem} \label{lem2.3a}
Suppose $\tilde u= \tilde u(t,x)$ has continuous second order derivatives. Then
\begin{align}
&| \langle r -t \rangle \partial_{tt} \tilde u (t,x) |
+| \langle r -t \rangle \partial_t \nabla  \tilde u (t,x) |
+| \langle r -t \rangle \Delta \tilde u (t,x) |  \notag \\
\lesssim &
| (\partial \Gamma^{\le 1} \tilde u)(t,x)| + (r+t) | (\square \tilde u)(t,x) |, \quad r=|x|, \,
t\ge 0;  \label{2.9a0}
\end{align}
and
\begin{align}
&| \langle r -t \rangle  \partial^2 \tilde u (t,x) |
\lesssim
| (\partial \Gamma^{\le 1} \tilde u)(t,x)| + (r+t) | (\square \tilde u)(t,x) |, \quad
\forall\, r\ge t/10,
 \, t\ge 1. \label{2.9a1}
\end{align}
Suppose $T_0\ge 1$ and $u \in C^{\infty}([1,T_0]\times \mathbb R^2)$ solves  \eqref{eq:we2d} with support in $|x|\le t+1$, $1\le t\le T_0$.
For any integer $l_0\ge 2$, there exists $\epsilon_1>0$ depending only on
$l_0$, such that if at some $1\le t\le T_0$,
\begin{align} \label{2.9a3}
\| (\partial \Gamma^{\le \lceil{\frac {l_0}2}\rceil +2} u)(t,\cdot) \|_{L_x^2(\mathbb R^2)} \le \epsilon_1,
\qquad  \text{( here $\lceil{z}\rceil = \min \{n\in \mathbb N:\, n\ge z\}$ )}
\end{align}
then for the same $t$, we have the $L^2$ estimate:
\begin{align} \label{2.9a4}
\| (\langle r -t \rangle \partial^2 \Gamma^{\le l_0} u) (t,\cdot) \|_{L_x^2(\mathbb R^2)}\lesssim \| (\partial \Gamma^{\le l_0+1} u )(t,\cdot) \|_{L_x^2(\mathbb R^2)}.
\end{align}
For any integer $l_1\ge 2$, there exists $\epsilon_2>0$ depending only on
$l_1$, such that if at some $1\le t\le T_0$,
\begin{align} \label{2.99a1}
\| (\partial \Gamma^{\le l_1 +1} u)(t,\cdot) \|_{L_x^2(\mathbb R^2)} \le \epsilon_2,
\end{align}
then for the same $t$, we have the point-wise estimate:
\begin{align} \label{2.99a2}
 |(\langle r -t  \rangle \partial^2 \Gamma^{\le l_1} u )(t,x) |
 \lesssim |  (\partial \Gamma^{\le l_1+1} u )(t,x)|,
 \qquad\forall\, r\ge t/10.
 \end{align}
 Moreover, we have
 \begin{align} \label{2.99aB2}
 \|  \partial \Delta \Gamma^{\le l_1-1} u \|_{L_x^2(|x|\le \frac 23 t)} 
 \lesssim t^{-2}\| (\partial \Gamma^{\le l_1 +1} u)(t,\cdot) \|_{L_x^2(\mathbb R^2)}.
 \end{align}
 \end{lem}
 \begin{rem}
 It also holds that
  \begin{align} \notag
 \|  \partial^3 \Gamma^{\le l_1-1} u \|_{L_x^2(|x|\le \frac 23 t)} 
 \lesssim t^{-2}\| (\partial \Gamma^{\le l_1 +1} u)(t,\cdot) \|_{L_x^2(\mathbb R^2)}.
 \end{align}
\end{rem}
\begin{proof}
All estimates except \eqref{2.99aB2} were proved in  Lemma 2.4 of \cite{CLX}. We now
sketch the proof of \eqref{2.99aB2}.  Applying \eqref{2.9a0} to $\tilde u
= \partial \Gamma^{\le l_1-1} u$ with $r\le \frac 23 t$, we get
\begin{align}
| \Delta \partial \Gamma^{\le l_1-1} u |
\lesssim \frac 1 t |\partial^2 \Gamma^{\le l_1 } u|
+ | \partial \square \Gamma^{\le l_1-1} u|.
\end{align}
By Lemma \ref{Lem2.3}, we have
\begin{align}
| \partial \square \Gamma^{\le l_1-1} u|
& \lesssim \sum_{a+b\le l_1-1} | \partial( \partial \Gamma^a u \partial^2 \Gamma^b u) |
\notag \\
& \lesssim 
|\partial^2 \Gamma^{\le l_1-1} u| | \partial^2 \Gamma^{\le l_1-1} u| +|\partial \Gamma^{\le l_1-1}u|
|\partial^3 \Gamma^{\le l_1-1} u|. \label{A2.21_0}
\end{align}
Note that 
\begin{align}
|\partial^3 \Gamma^{\le l_1-1} u|
&\lesssim | \underbrace{\partial_{tt} \partial \Gamma^{\le l_1-1} u }_{\text{$\partial_t$ appears
twice or more}} |+ \sum_{1\le i_1,i_2\le 2} |\underbrace{
\partial \partial_{i_1} \partial_{i_2} \Gamma^{\le l_1-1} u}_{\text{$\partial_t$
appears at most once}}|  \notag \\
& \lesssim | \partial \square \Gamma^{\le l_1-1} u|
+ |\partial \tilde \partial^2 \Gamma^{\le l_1-1} u |, \label{A2.21_1}
\end{align}
where we have denoted $\tilde \partial = (\partial_1, \partial_2)$.  By using the smallness
of the pre-factor $\|\partial \Gamma^{\le l_1-1} u\|_{\infty}$ and
\eqref{A2.21_1},  we then derive from \eqref{A2.21_0}
\begin{align}
|\partial \square \Gamma^{\le l_1-1} u |\lesssim 
|\partial^2 \Gamma^{\le l_1-1} u| | \partial^2 \Gamma^{\le l_1-1} u| +|\partial \Gamma^{\le l_1-1}u|
|\partial \tilde \partial^2 \Gamma^{\le l_1-1} u|. \label{A2.21_2}
\end{align}

Clearly by Sobolev embedding $H^1(\mathbb R^2) \hookrightarrow L^4(\mathbb R^2)$, 
we get (below denote $X=\| (\partial \Gamma^{\le l_1 +1} u)(t,\cdot) \|_2$)
\begin{align}
&\| \langle r-t\rangle \partial^2 \Gamma^{\le l_1-1} u
\|_4  \lesssim \| \partial^{\le 1} ( \langle r -t \rangle  \partial^2 \Gamma^{\le l_1-1} u)
\|_2 \lesssim X;  \\
& \| \langle r -t \rangle^{\frac 12}  \partial \Gamma^{\le l_1-1} u\|_{\infty}
 \lesssim X,
\qquad (\text{by Lemma \ref{lem2.5A}}). 
\end{align}
By using a smooth cut-off function localized to $|x| \le \frac 23 t$, we then derive
\begin{align}
\| \Delta \partial \Gamma^{\le l_1-1} u\|_{L_x^2(|x|\le \frac 23 t)} \lesssim t^{-\frac 32} X.
\end{align}
It follows that (recall $\tilde \partial=(\partial_1, \partial_2)$)
\begin{align}
\| \tilde \partial^2 \partial \Gamma^{\le l_1-1} u\|_{L_x^2(|x|\le \frac 23 t)} \lesssim t^{-\frac 32}X.
\end{align}
Plugging this estimate into \eqref{A2.21_2},  we then 
obtain the estimate \eqref{2.99aB2}.
\end{proof}

\begin{lem} \label{lem2.5A}
For any $f \in C_c^{\infty}(\mathbb R^2)$, we have
\begin{align}
&\langle |x_0| -t \rangle^{\frac 12} |f(x_0)| \lesssim \|f\|_2+\| \langle |x|-t\rangle \nabla f
\|_2 + \| \langle |x|-t\rangle \partial_1 \partial_2 f\|_2, \quad \forall\, x_0\in \mathbb R^2,
\, t\ge 0;  \label{2.26A}\\
& \| \langle |x| -t \rangle \partial f \|_{\infty}
\lesssim \| \langle |x|-t\rangle \partial f \|_2 + \| \langle |x| -t\rangle \partial^2 f \|_2
+ \| \langle |x| -t \rangle \partial^3 f \|_2, \, \quad\forall\, t\ge 0.
\label{2.26B}
\end{align}
It follows that
\begin{align}
&\| f \|_{L_x^{\infty}(\mathbb R^2)} \lesssim
\langle t \rangle^{-\frac 12}
(\|f\|_2 + \| \langle |x| -t\rangle \nabla \tilde \Gamma^{\le 1 } f \|_2), \qquad\forall\, t\ge 0,
\label{2.27A}
\end{align}
where $\tilde \Gamma =(\partial_1, \partial_2, \partial_{\theta})$.
\end{lem}
\begin{proof}
See Lemma 2.5 of \cite{CLX}.
\end{proof}

\begin{lem}[Decay estimates] \label{lem2.6}
Suppose $T_0\ge 2$ and $u \in C^{\infty}([2,T_0]\times \mathbb R^2)$ solves  \eqref{eq:we2d} with support in $|x|\le t+1$, $2\le t\le T_0$. Suppose $J\ge 3$ and
\begin{align}
E_J=E_J(u(t,\cdot) ) = \| (\partial \Gamma^{\le J} u )(t,\cdot) \|_2^2 \le \tilde \epsilon,
\end{align}
where $\tilde \epsilon>0$ is sufficiently small. Then we have the following decay estimates:
\begin{align}
& t^{\frac 12} \| \partial \Gamma^{\le J-2} u \|_{L_x^{\infty}}
+t^{\frac 12}\| \langle |x|-t\rangle \partial^2 \Gamma^{\le J-3} u \|_{L_x^{\infty}(|x|>\frac{t}{10})}
+\| \langle |x|-t\rangle \partial^2 \Gamma^{\le J-3} u \|_{L_x^{\infty}}
\lesssim \;E_J^{\frac 12}; \label{2.30A}\\
& \| \partial^2\Gamma^{\le J-3} u \|_{L_x^{\infty}(|x|<\frac t2)}
\lesssim t^{-\frac 32} E_J^{\frac 12}; \qquad \label{2.30AA}\\
& 
\| \langle |x|-t\rangle \partial^2 \Gamma^{\le J-3} u \|_{L_x^{\infty}}
\lesssim t^{-\frac 12} E_J^{\frac 12} ;
\qquad \label{2.30AB}\\
&\| \frac{T \Gamma^{\le J-2} u} {\langle |x|-t \rangle} \|_{L_x^{\infty}}
+ \| T \partial \Gamma^{\le J-3} u \|_{L_x^{\infty}}
\lesssim t^{-\frac 32} E_J^{\frac 12};  \label{2.30B} \\
& \| \frac{ T \Gamma^{\le J-1} u } { \langle |x| -t \rangle }
\|_{L_x^2}  +\| T \partial \Gamma^{\le J-1} u \|_{L_x^2} \lesssim
t^{-1} E_J^{\frac 12}. \label{2.30C}  
\end{align}
More generally, for any integer $J_1\ge 1$, we have
\begin{align} \label{2.30D}
\| \frac{T\Gamma^{\le J_1} u } {\langle |x| -t \rangle }
\|_{L_x^2} \lesssim t^{-1} \| \partial \Gamma^{\le J_1+1} u \|_{L_x^2}.
\end{align}
\end{lem}
\begin{proof}
We shall take $\tilde \epsilon$ sufficiently small so that Lemma
\ref{lem2.3a} can be applied.  The estimate \eqref{2.30A} follows from Lemma \ref{lem2.5A} and Lemma \ref{lem2.3a}.
To derive the estimate \eqref{2.30AA}, we choose $\psi \in C_c^{\infty}(\mathbb R^2)$
such that $\psi(z) \equiv 1$ for $|z|\le 0.5$ and $\psi(z) \equiv 0$ for $|z|\ge 0.52$. Applying
the interpolation inequality $\|\tilde v\|_{\infty} \lesssim \|\tilde v \|_2^{\frac 12}
 \|\Delta \tilde v \|_2^{\frac 12}$ with $\tilde v(x) =\psi(\frac x t) \partial^2 
 \Gamma^{\le J-3} u$, we obtain
 \begin{align}
 \| \psi(\frac x t) \partial^2 \Gamma^{\le J-3} u 
 \|_{\infty} \lesssim \| \psi(\frac x t) \partial^2 \Gamma^{\le J-3} u 
 \|_2^{\frac 12}
 \| \Delta ( \psi(\frac x t) \partial^2 \Gamma^{\le J-3} u  ) \|_2^{\frac 12}.
 \end{align}
 By Lemma \ref{lem2.3a}, it is not difficult to check that
\begin{align}
 \| \Delta ( \psi(\frac x t) \partial^2 \Gamma^{\le J-3} u  ) \|_2 \lesssim t^{-2} E_J^{\frac 12},
 \quad  \| \psi(\frac x t) \partial^2 \Gamma^{\le J-3} u   \|_2  \lesssim t^{-1}
 E_J^{\frac 12}.
 \end{align}
 The estimate \eqref{2.30AA} then follows. For the estimate \eqref{2.30AB} we only
 need to examine the regime $|x| \ge t/2$. But this follows from Lemma \ref{lem2.3a} and
 \ref{lem2.5A}.

For \eqref{2.30B}, we note that the case $|x|\le \frac t2$ follows from
\eqref{2.30AA}.
On the other hand, for $|x|>\frac t2$ we denote $\tilde u=\Gamma^{\le J-2} u$ and estimate
$\| \frac {T_1 \tilde u} {\langle |x|-t \rangle} \|_{L_x^{\infty}(|x|>\frac t2)}$ (the
estimate for $T_2$ is similar). Recall that
\begin{align}
T_1 \tilde u & = \omega_1 \partial_t \tilde u + \partial_1 \tilde u
= \omega_1 (\partial_t + \partial_r ) \tilde u- \frac{\omega_2} r \partial_{\theta} \tilde u\notag \\
& = \omega_1 \frac 1 {t+r} ( 2 L_0 \tilde u - (t-r) \partial_-\tilde u) -\frac{\omega_2} r \partial_{\theta}
\tilde u.
\end{align}
Clearly for $r=|x|\ge \frac t2$,
\begin{align}
\left| \frac{ T_1 \tilde u } { \langle r -t \rangle}
\right| &\lesssim \frac1 t \Bigl( \left| \frac{L_0 \tilde u} {\langle r -t\rangle } \right| + |\partial \tilde u|
\Bigr)
+ \left| \frac{\partial_{\theta} \tilde u} {r \langle r-t \rangle } \right| \notag \\
& \lesssim t^{-1} \cdot t^{-\frac 12} \| \partial \Gamma^{\le 2} \tilde u\|_2
+ t^{-\frac 32}+t^{-1} \cdot t^{-\frac 12} \| \partial \Gamma^{\le 1} \partial_{\theta} \tilde u\|_2 \lesssim
t^{-\frac 32} E_J^{\frac 12},
\end{align}
where in the second last step we used Lemma \ref{lem:S} (for the term $|\partial \tilde u|$
we use \eqref{2.30A}).  The estimates for \eqref{2.30C}--\eqref{2.30D} is similar. We omit the details.
\end{proof}

\section{Proof of Theorem \ref{thm:main} }
In this section we carry out the proof of Theorem \ref{thm:main}.
Write $v= \Gamma^{\alpha} u$. By Lemma \ref{Lem2.3} we have
\begin{align}
\square v  &= \sum_{\alpha_1+ \alpha_2 \le \alpha}
g^{kij}_{\alpha_1,\alpha_2} \partial_k \Gamma^{\alpha_1} u
\partial_{ij} \Gamma^{\alpha_2} u \label{3.0A}\\
&=
g^{kij}\partial_k v
\partial_{ij} u  +
g^{kij} \partial_k  u
\partial_{ij} v
+\sum_{\substack{\alpha_1<\alpha, \alpha_2<\alpha;\\ \alpha_1+ \alpha_2 \le \alpha}}
g^{kij}_{\alpha_1,\alpha_2} \partial_k \Gamma^{\alpha_1} u
\partial_{ij} \Gamma^{\alpha_2} u.
\end{align}

Choose $p(t,r) = q(r-t)$, where
\begin{align}
&q(s) = \int_0^s \langle \tau\rangle^{-1}  \bigl(\log ( 2+\tau^2) \bigr)^{-2} d\tau, \quad s\in \mathbb R; \\
&-\partial_t p = \partial_r p = q^{\prime}(r-t)
= \langle r -t \rangle^{-1} \bigl( \log (2+(r-t)^2)  \bigr)^{-2}.
\end{align}
Multiplying both sides of \eqref{3.0A} by $e^p \partial_t v$, we obtain
\begin{align*}
   \text{LHS}&=\int e^p \pa_{tt}v\pa_{t}v -\int e^p \Delta v\pa_{t}v
   =\int e^p \pa_{tt} v\pa_{t}v +\int e^p\na v\cdot\na\pa_{t}v+\int e^p\na v\cdot\na p\pa_{t}v \\
   &=\frac12\frac{d}{dt}\int e^p(\pa v)^2 -\frac12\int e^{p}|\pa v|^2 p_{t}+\int e^p\na v\cdot\na p\pa_{t}v \\
   &=\frac12\frac{d}{dt}\| e^\frac{p}{2}\pa v\|_{L^2}^2 +\frac 12
   \int e^p q^{\prime} \cdot \Bigl( |\partial_+ v|^2+ \frac {|\partial_{\theta} v|^2}{r^2}  \Bigr)
   =\frac12\frac{d}{dt}\| e^\frac{p}{2}\pa v\|_{L^2}^2 +\frac 12
   \int e^p q^{\prime} |T v|^2.
\end{align*}
We shall sum over $|\alpha|\le m_0$, where $ m_0=m-1$ or $m$ is a running parameter. 
To simplify the notation in the subsequent nonlinear estimates, we introduce the following
terminology.

\noindent
\textbf{Notation}.  For a quantity $X(t)$, we shall write $X(t)= \mathrm{OK}$  if $X(t)$ can be written as
\begin{align} \label{No3.5}
X(t)= \frac d {dt} X_1 (t)+X_2(t) +X_3(t),
\end{align}
where  (below $\alpha_0>0$ is some constant)
\begin{align}
|X_1(t)| \ll \| (\partial \Gamma^{\le m_0}  u)(t,\cdot) \|_{L_x^2(\mathbb R^2)}^2,
\quad |X_2(t)| \ll    \sum_{|\alpha| \le m_0} \int e^p q^{\prime} |(T \Gamma^{\alpha} u)(t,x) |^2 dx,
\quad |X_3(t)| \lesssim t^{-1-\alpha_0}.
\end{align}
In yet other words, the quantity $X$ will be controllable  if either it can be absorbed into
the energy, or can be controlled by the weighted $L^2$-norm of the good
unknowns  from the Alinhac weight,  or it is integrable in time.

We now proceed with the nonlinear estimates. We shall discuss several cases.

\subsection{The case $\alpha_1<\alpha$ and $\alpha_2<\alpha$}
Since $g^{kij}_{\alpha_1,\alpha_2}$ still
satisfies the null condition, by \eqref{2.10A} we have
\begin{align}
  &\sum_{\substack{\alpha_1<\alpha,  \alpha_2 <\alpha \\ \alpha_1+\alpha_2\le \alpha}}
g^{kij}_{\alpha_1,\alpha_2} \partial_k \Gamma^{\alpha_1} u
\partial_{ij} \Gamma^{\alpha_2} u    \notag \\
= &\sum_{\substack{\alpha_1<\alpha,  \alpha_2 <\alpha \\ \alpha_1+\alpha_2\le \alpha}}
g^{kij}_{\alpha_1,\alpha_2}
(T_k \Gamma^{\alpha_1} u \partial_{ij} \Gamma^{\alpha_2} u
-\omega_k \partial_t \Gamma^{\alpha_1 } u T_i \partial_j \Gamma^{\alpha_2} u
+ \omega_k \omega_i \partial_t \Gamma^{\alpha_1} u T_j
\partial_t \Gamma^{\alpha_2} u ).
\end{align}

\texttt{Estimate of $\| T_k \Gamma^{\alpha_1} u  \partial^2 \Gamma^{\alpha_2} u\|_2$}.
If $|\alpha_1|\le |\alpha_2|$,  then by Lemma \ref{lem2.6} we have
\begin{align}
 \| T_k \Gamma^{\alpha_1} u  \partial^2 \Gamma^{\alpha_2} u\|_2
 \lesssim
 \| \frac {T_k \Gamma^{\alpha_1} u } {\langle r-t\rangle}
 \|_{\infty} \cdot
 \| \langle r-t\rangle \partial^2 \Gamma^{\alpha_2}u \|_2
 \lesssim  t^{-\frac 32} E_{\Flor{\frac {m_0}2}+2}^{\frac12} E_{m_0}^{\frac 12}.
 \end{align}
 If $|\alpha_1|>|\alpha_2|$, then we have
 \begin{align}
 \| T_k \Gamma^{\alpha_1} u  \partial^2 \Gamma^{\alpha_2} u\|_2
 \lesssim
 \| \frac {T_k \Gamma^{\alpha_1} u } {\langle r-t\rangle}
 \|_{2} \cdot
 \| \langle r-t\rangle \partial^2 \Gamma^{\alpha_2}u \|_{\infty}
 \lesssim  t^{-\frac 32} E_{\Flor{\frac {m_0}2}+3}^{\frac12} E_{m_0}^{\frac 12}.
 \end{align}

\texttt{Estimate of $\| \partial \Gamma^{\alpha_1} u T \partial \Gamma^{\alpha_2} u\|_2$}.
If $|\alpha_1| \le |\alpha_2|$, we have
\begin{align}
\| \partial \Gamma^{\alpha_1} u T \partial \Gamma^{\alpha_2} u\|_2
\lesssim \| \partial \Gamma^{\alpha_1} u \|_{\infty}
\cdot \| T \partial \Gamma^{\alpha_2} u \|_2
\lesssim t^{-\frac 32} E_{\Flor{\frac {m_0}2}+3}^{\frac12} E_{m_0}^{\frac 12}.
\end{align}
If $|\alpha_1| > |\alpha_2|$, we have
\begin{align}
\| \partial \Gamma^{\alpha_1} u T \partial \Gamma^{\alpha_2} u\|_2
\lesssim \| \partial \Gamma^{\alpha_1} u \|_{2}
\cdot \| T \partial \Gamma^{\alpha_2} u \|_{\infty} \lesssim  
t^{-\frac 32} E_{\Flor{\frac {m_0}2}+3}^{\frac12} E_{m_0}^{\frac 12}.
\end{align}

Collecting the estimates, we have proved
\begin{align} \label{4.13U}
  &\| \sum_{\substack{\alpha_1<\alpha,  \alpha_2 <\alpha \\ \alpha_1+\alpha_2\le \alpha}}
g^{kij}_{\alpha_1,\alpha_2} \partial_k \Gamma^{\alpha_1} u
\partial_{ij} \Gamma^{\alpha_2} u    \|_2 \lesssim  
t^{-\frac 32} E_{\Flor{\frac {m_0}2}+3}^{\frac12} E_{m_0}^{\frac 12}.
\end{align}

\subsection{The case $\al_{2}=\al$.}
Noting that $g^{kij}_{0,\alpha}=g^{kij}$, we have
\begin{align}
 \int g^{kij}\pa_{k}u\pa_{ij}v\pa_{t}v e^{p}
 &= \OK \underbrace{- \int g^{kij}\pa_{jk}u\pa_{i}v\pa_{t}v e^{p}}_{I_1}\underbrace{-\int g^{kij}\pa_{k}u\pa_{i}v\pa_{t}v \pa_{j}(e^{p})}_{I_2}-\int g^{kij}\pa_{k}u\pa_{i}v\pa_{tj}v e^{p}.
\end{align}
Here in the above, the term ``OK" is zero if $\partial_j =\partial_1$ or $\partial_2$.
This term is nonzero when $\partial_j = \partial_t$, i.e. we should absorb it into the energy when
integrating by parts in the time variable.

Further integration by parts gives
\begin{align}
  -\int g^{kij}\pa_{k}u\pa_{i}v\pa_{tj}v e^{p}
  &=\OK+ \underbrace{\int g^{kij}\pa_{tk}u\pa_{i}v\pa_{j}v e^{p}}_{I_3}+\underbrace{\int g^{kij}\pa_{k}u\pa_{i}v\pa_{j}v \pa_{t}(e^{p})}_{I_4}+\int g^{kij}\pa_{k}u\pa_{it}v\pa_{j}
  v e^{p}.
  \end{align}
\begin{align}
  \int g^{kij}\pa_{k}u\pa_{it}v\pa_{j}v e^{p}&=\OK \underbrace{-\int g^{kij}\pa_{ik}u\pa_{t}v\pa_{j}v e^{p}}_{I_5}\underbrace{-\int g^{kij}\pa_{k}u\pa_{t}v\pa_{j}v \pa_{i}(e^{p})}_{I_6}-\int g^{kij}\pa_{k}u\pa_{t}v\pa_{ij}v e^{p}.
\end{align}
It follows that
$$2\int g^{kij}\pa_{k}u\pa_{ij}v\pa_{t}v e^{p}=(I_1+I_3+I_5)+(I_2+I_4+I_6) +\OK.$$

Observe that if $\varphi=\pa_{k}u$ or $\varphi=e^{p}$, then
\begin{align}\label{varphi1}
   &-\pa_{j}\varphi\pa_{i}v\pa_{t}v+\pa_{t}\varphi\pa_{i}v\pa_{j}v-\pa_{i}\varphi\pa_{t}v\pa_{j}v
   \notag\\
  =&-T_{j}\varphi\pa_{i}v\pa_{t}v+\omega_{j}\pa_{t}\varphi\pa_{i}v \pa_{t} v+\pa_{t}\varphi\pa_{i}v\pa_{j}v-T_{i}\varphi\pa_{t}v\pa_{j}v+\omega_{i}\pa_{t}\varphi\pa_{t}v T_{j}v-\omega_{i}\omega_{j}\pa_{t}\varphi(\pa_{t}v)^2 \notag \\
  =&-T_{j}\varphi\pa_{i}v\pa_{t}v+\pa_{t}\varphi\pa_{i}v T_{j} v-T_{i}\varphi\pa_{t}v\pa_{j}v+\omega_{i}\pa_{t}\varphi\pa_{t}v T_{j}v-\omega_{i}\omega_{j}\pa_{t}\varphi(\pa_{t}v)^2 \notag\\
  =&-T_{j}\varphi\pa_{i}v\pa_{t}v+\pa_{t}\varphi T_{i}v T_{j} v-T_{i}\varphi\pa_{t}v\pa_{j}v-\omega_{i}\omega_{j}\pa_{t}\varphi(\pa_{t}v)^2.
 \end{align}
By \eqref{varphi1} and rewriting $\partial_t \varphi = \partial_k \partial_t u
=T_k \partial_t u -\omega_k \partial_{tt} u$, we have
\begin{align}
I_1+I_3+I_5=&\;\int g^{kij}(
-T_{j}\partial_k u \pa_{i}v\pa_{t}v+\pa_{t}\partial_k u T_{i}v T_{j} v-T_{i}\partial_k u\pa_{t}v\pa_{j}v-\omega_{i}\omega_{j}T_k \partial_t u(\pa_{t}v)^2) e^pdx.
\end{align}
By Lemma \ref{lem2.6}, we have $\|T \partial u\|_{\infty} \lesssim t^{-\frac 32} E_3^{\frac 12}$
and $\| \langle r-t\rangle \partial^2 u \|_{\infty} \lesssim t^{-\frac 12} E_3^{\frac 12} $. Clearly then
\begin{align}
\int_{\text{$r<\frac t2$ or $r>2t$} } |\partial^2 u| |T v|^2 dx\lesssim t^{-\frac 32} 
E_3^{\frac 12} E_{m_0},
\quad \int_{r \sim t } |\partial^2 u | |Tv|^2 dx
\ll \int e^p q^{\prime} |Tv|^2 dx.
\end{align}
It follows that
\begin{align}
I_1+I_3+I_5 =\OK.
\end{align}
Plugging $\varphi=e^{p}$ in \eqref{varphi1} and noting that $T_j (e^p)=0$, we  have
\begin{align*}
    I_2+I_4+I_6
    &=\int g^{kij}\pa_k u\Big(-T_j(e^p)\pa_iv \pa_tv- T_i(e^p)\pa_tv\pa_jv-\omega_i\omega_j(\pa_tv)^2\pa_t(e^p)+\pa_t(e^p)T_iv T_jv \Big) \notag \\
  &=   \int g^{kij}\left(  -T_k u\cdot \omega_i\omega_j (\pa_tv)^2 \pa_t(e^p) +\partial_k u\pa_t(e^p)T_i v T_j v \right).
\end{align*}
By Lemma \ref{lem2.6} we have $||Tu| | \partial_t (e^p)|\lesssim t^{-\frac 32} E_3^{\frac 12}$.
Clearly
\begin{align}
\|\partial u \partial_t (e^p)\|_{L_x^{\infty}(r<\frac t2,\, \text{or } r>2t)} \lesssim t^{-\frac 32}E_3^{\frac 12},
\quad \int_{r\sim t} |\partial u \partial_t (e^p)| |Tv|^2 dx \ll \int e^p q^{\prime} |Tv|^2 dx.
\end{align}
Thus
\begin{align*}
  I_{2}+I_{4}+I_{6} = \OK.
\end{align*}
This concludes the case $\alpha_2=\alpha$. 

\subsection{The case $\al_{1}=\al$, $\al_{2}=0$}
By \eqref{2.10A}, we have
\begin{align*}
 \int g^{kij}\pa_{k}v\pa_{ij}u\pa_{t}v e^{p}
  =&\int g^{kij} (T_{k}v\pa_{ij}u-\omega_{k}\pa_{t}vT_{i}\pa_{j}u+\omega_{k}\omega_{i}\pa_{t}v T_{j}\pa_{t}u)\pa_{t}v e^{p}\\
  =&\int g^{kij} (T_{k}vT_{i}\pa_{j}u-\omega_{i}T_{k}vT_{j}\pa_{t}u+\omega_{i}\omega_{j}T_{k}v\pa_{tt}u-\omega_{k}\pa_{t}vT_{i}\pa_{j}u+\omega_{k}\omega_{i}\pa_{t}v T_{j}\pa_{t}u)\pa_{t}v e^{p}.
\end{align*}
By Lemma \ref{lem2.6}, all terms containing $T\partial u$ decay as $O(t^{-\frac 32} E_3^{\frac 12}
E_{m_0})$.
Thus
\begin{equation}\label{eq:al1-00}
  \int g^{kij} \pa_{k}v\pa_{ij}u\pa_{t}v e^{p}= \OK+
  \underbrace{\int g^{kij} \omega_{i}\omega_{j}T_{k}v\pa_{tt}u\pa_{t}v e^{p}}_{=:Y_1}.
\end{equation}
We now discuss two cases.

Case 1: $m_0=m-1$.  By Lemma \ref{lem2.6}, we have
\begin{align}
\| \frac {Tv} {\langle |x| -t \rangle } \|_{L_x^2}
\lesssim t^{-1} E_{m_0+1}^{\frac 12} = t^{-1} E_m^{\frac 12},
\qquad \| \langle |x|-t\rangle \partial_{tt} u \|_{L_x^{\infty}}
\lesssim t^{-\frac 12} E_3^{\frac 12}.
\end{align}
Thus
\begin{align}
|Y_1| \lesssim t^{-\frac 32} E_3^{\frac 12} E_m^{\frac 12} E_{m_0}^{\frac 12}.
\end{align}

Case 2: $m_0=m$. By using Cauchy-Schwartz, we have
\begin{align}
|Y_1| &\le \OK + \mathrm{const} \cdot \int \frac 1 {q^{\prime}}|\partial_{tt} u |^2
|\partial_t v|^2 dx \notag \\
& \le  \OK+ \frac 1 t E_3 E_m.
\end{align}

Collecting all the estimates and assuming the norm of the initial data is
sufficiently small, we then obtain for some small constants $\epsilon_3>0$, $\epsilon_4>0$,
\begin{align}
\sup_{t\ge 2} E_{m-1} (u(t,\cdot) ) \le \epsilon_3 \ll 1, \qquad
\sup_{t\ge 2} \frac {E_m(u(t,\cdot) )} {t^{\epsilon_4} } \le 1.
\end{align}
This concludes the proof of Theorem \ref{thm:main}.

\bibliographystyle{abbrv}

\begin{thebibliography}{}
\bibitem{Agemi00} R. Agemi. Global existence of nonlinear elastic waves, Invent. Math., 142: 225-250, 2000.

\bibitem{Alinh93} S. Alinhac. Temps de vie des solutions r\'eguli\`eres des \'equations d'Euler compressible axisym\'etriques en dimension deux, Invent. Math., 111: 627-670, 1993.



\bibitem{Alinh01_1} S. Alinhac. The null condition for quasilinear wave equations in two space dimensions I, Invent. Math., 145: 597-618, 2001.

\bibitem{Alinh01_2} S. Alinhac. The null condition for quasilinear wave equations in two space dimensions II, Amer. J. Math., 123: 1071-1101, 2001.



\bibitem{Alinh_bk}S. Alinhac. Geometric Analysis of Hyperbolic Differential Equations: An Introduction., London Math. Soc. Lecture Note Ser., vol. 374, Cambridge University Press, Cambridge, 2010.


\bibitem{Chrd86} D. Christodoulou. Global solutions of nonlinear hyperbolic equations for small initial data, Comm. Pure Appl. Math., 39: 267-282, 1986.


\bibitem{Hor86}L. H\"ormander. The lifespan of classical solutions of nonlinear hyperbolic equations, Pseudodifferential Operators, Lecture Notes in Math., vol. 1256, Oberwolfach, 1986, Springer, Berlin, 214-280, 1987.


\bibitem{Hor97}L. H\"ormander. Lectures on Nonlinear Hyperbolic Differential Equations Math\'ematiques \& Applications (Berlin) [Mathematics \& Applications], vol. 26, Springer-Verlag, Berlin, 1997.



\bibitem{Hosh06} A. Hoshiga. The existence of global solutions to systems of quasilinear wave equations with quadratic nonlinearities in 2-dimensional space., Funkcial. Ekvac., 49: 357-384, 2006.

\bibitem{John90}F. John. Nonlinear Wave Equations, Formation of Singularities University Lecture Series, vol. 2, American Mathematical Society, Providence, RI 1990, Seventh Annual Pitcher Lectures delivered at Lehigh University, Bethlehem, Pennsylvania, April 1989.



\bibitem{SK96}S. Klainerman, T.C. Sideris. On almost global existence for nonrelativistic wave equations in 3D., Comm. Pure Appl. Math., 49: 307-321, 1996.



\bibitem{LiZhou16}T. Li, Y. Zhou. Nonlinear Wave Equations (in Chinese)., Series in Contemporary Mathematics, vol. 1, Shanghai Scientific \& Technical Publishers, 2016.




\bibitem{CLX}X.Y. Cheng, D. Li and  J. Xu. Uniform boundedness of highest norm for 2D quasilinear wave. Preprint arXiv: 2104.10019.


\bibitem{Lei16}
Z. Lei. Global wellposedness of incompressible elastocdynamics in 2D.
Comm. Pure Appl. Math. 69 (2016), 2072--2106.

\bibitem{CLM18}
Y. Cai, Z. Lei and N. Masmoudi. Global wellposedness for 2D nonlinear
wave equations without compact support. J. Math. Pures Appl. 114 (2018), 211-234.

\bibitem{Cai21}
Y. Cai. Uniform bound of the highest-order energy of the 2D incompressible
elastodynamics. Preprint arXiv: 2010.08718. 

\bibitem{Dong21}
S. Dong, P. LeFloch, and Z. Lei. The top-order energy of quasilinear
wave equations in two space dimensions is uniformly bounded.
Preprint arXiv: 2103.07867.


\bibitem{HY20}
F. Hou, H. Yin. Global small data smooth solutions of 2-D null-form wave equations with non-compactly supported initial data. J. Differential Equations 268 (2020), no. 2, 490?512.

\bibitem{Zh19}
 D. Zha. Global and almost global existence for general quasilinear wave equations in two space dimensions. J. Math. Pures Appl. (9) 123 (2019), 270?299.
 
 \bibitem{Ka17}
 S. Katayama. Global Solutions and the Asymptotic Behavior for Nonlinear Wave Equations with Small Initial Data, MSJ Memoirs, vol. 36. Mathematical Society of Japan, Tokyo (2017)

\bibitem{MetSog06}J. Metcalfe, C.D. Sogge. Long-time existence of quasilinear wave equations exterior to star-shaped obstacles via energy methods
SIAM J. Math. Anal., 38: 188-209, 2006.


\bibitem{MetSog07}J. Metcalfe, C.D. Sogge. Global existence of null-form wave equations in exterior domains., Math. Z., 256: 521-549, 2007.

\bibitem{PZ16}
W. Peng and D. Zha. A note on quasilinear wave equations in two space dimensions II: Almost global existence of classical solutions. Journal of Mathematical Analysis and Applications, 439(1),
pp.419--435, 2016.







\bibitem{Sid97}T.C. Sideris. Delayed singularity formation in 2D compressible flow., Amer. J. Math., 119: 371-422, 1997.


\bibitem{Sid00}T.C. Sideris. Nonresonance and global existence of prestressed nonlinear elastic waves
Ann. of Math. 151(2): 849-874, 2000.

\bibitem{ST01}
T.C. Sideris and S.-Y. Tu.
Global existence for systems of nonlinear
wave equations in 3D with multipler speeds.
SIAM J. Math. Anal., 33 (2001), 477--488.



\bibitem{Sogbk}C.D. Sogge. Lectures on Non-linear Wave Equations (2nd ed.), International Press, Boston, MA, 2008.

\bibitem{WangYu14}C. Wang, X. Yu. Global existence of null-form wave equations on small asymptotically Euclidean manifolds., J. Funct. Anal., 266: 5676-5708, 2014.

\bibitem{Yang13}S. Yang. Global solutions of nonlinear wave equations in time dependent inhomogeneous media., Arch. Ration. Mech. Anal., 209: 683-728, 2013.


\bibitem{Yang16}S. Yang. On the quasilinear wave equations in time dependent inhomogeneous media., Journal of Hyperbolic Differential Equations., 13(2): 273-330, 2016.

\bibitem{Zha16}D. Zha. A note on quasilinear wave equations in two space dimensions., Discrete Contin. Dyn. Syst., Ser. A, 36: 2855-2871, 2016.






\end{thebibliography}

\end{document}